\documentclass[11pt,reqno]{amsart}
\usepackage[left=0.9in,top=0.9in,right=0.9in,bottom=0.9in]{geometry}
\usepackage{amsaddr}

\usepackage{amsfonts,amsmath,amsthm,amssymb,latexsym,mathrsfs,stmaryrd}
\usepackage{hyperref}
\usepackage{mathtools}
\usepackage{graphicx}
\usepackage{subcaption}

\usepackage{tikz}\usetikzlibrary{cd,fit,matrix,arrows,decorations.pathmorphing}
\tikzset{commutative diagrams/.cd}

\numberwithin{equation}{section}

\newtheorem{theorem}{Theorem}[section]
\newtheorem{corollary}[theorem]{Corollary}
\newtheorem{lemma}[theorem]{Lemma}
\newtheorem{proposition}[theorem]{Proposition}

\theoremstyle{definition}
\newtheorem{definition}[theorem]{Definition}
\newtheorem{definition-theorem}[theorem]{Definition-Theorem}

\newtheorem{question}[theorem]{Question}

\theoremstyle{remark}

\newtheorem*{remark*}{Remark}

\usepackage{enumitem}
\setenumerate[1]{label=\textup{(\alph*)}}
\setenumerate[2]{label=\textup{(\roman*)}}

\newcommand\Z{{\mathbb Z}}

\newcommand\Q{{\mathbb Q}}
\newcommand\C{{\mathbb C}}

\newcommand\Fp{{{\mathbb F}_p}}
\newcommand\Fq{{\mathbb F}_q}

\newcommand\Zp{{\mathbb Z}_p}

\DeclareMathOperator{\im}{im}

\DeclareMathOperator{\Hom}{Hom}

\DeclareMathOperator{\Aut}{Aut}

\DeclareMathOperator{\GL}{GL}

\DeclareMathOperator{\Mat}{Mat}

\newcommand\subeq{\subseteq}

\newcommand\hhat{\widehat} 
\newcommand\onto{\twoheadrightarrow}

\DeclarePairedDelimiter{\abs}{\lvert}{\rvert}

\DeclarePairedDelimiter{\parens}{\lparen}{\rparen}

\newcommand{\Zhat}{\hhat{Z}}
\newcommand{\Mod}{\mathbf{Mod}}

\newcommand{\Surj}{\mathrm{Surj}}

\DeclareMathOperator*{\Prob}{Prob}

\newcommand{\cok}{\mathrm{cok}}
\newcommand{\Fl}{\mathbf{Fl}}
\newcommand{\bfcok}{\mathbf{cok}}
\newcommand{\calFl}{\mathcal{F}l}
\newcommand{\itker}{\mathit{ker}}
\newcommand{\itim}{\mathit{im}}
\newcommand{\Nilp}{\mathrm{Nilp}}

\begin{document}
\title[Random matrix products and flag Cohen--Lenstra]{Cokernels of random matrix products and flag Cohen--Lenstra heuristic}

\author{Yifeng Huang}
\address{Dept.\ of Mathematics, University of British Columbia}
\email{huangyf@math.ubc.ca}



\begin{abstract}
    In \cite{nguyenvanpeski}, Nguyen and Van Peski raised the question of whether the surjective flag of $\Zp$-modules modeled by $\cok(M_1\cdots M_k)\onto \dots\onto \cok(M_1)$ for independent random matrices $M_1,\dots,M_k\in \Mat_n(\Zp)$ satisfies the Cohen--Lenstra heuristic. We answer the question affirmatively when $M_1,\dots,M_k$ follow the Haar measure, and our proof demonstrates how classical ideas in Cohen--Lenstra heuristic adapt naturally to the flag setting. We also prove an analogue for non-square matrices.
\end{abstract}

\maketitle

\section{Introduction}

We start by introducing our notation for surjective flags of $\Zp$-modules, which is our main subject of investigation.

\subsection{Notation and terminology}\label{subsec:notn}
Fix $k\in \Z_{\geq 1}$ and a prime $p$. Let $\Mod_{\Zp}$ denote the category of finitely generated $\Zp$-modules. We refer to a diagram of surjections $G_k\overset{\phi_{k-1}}{\onto} \dots \overset{\phi_1}{\onto} G_1 (\overset{\phi_0}{\onto} G_0:=0)$ in $\Mod_{\Zp}$ as a $k$-\textbf{surjective flag} of (finitely generated) $\Zp$-modules. We denote by $\Fl_k=\Fl_k(\Zp)$ the set of $k$-surjective flags of $\Zp$-modules up to isomorphism. For $\mathbf{G}=(G_k\onto\dots\onto G_1)\in \Fl_k$, let $\Aut(\mathbf{G})$ denote the automorphism of the flag $\mathbf{G}$. See also \cite[Def.~18, 19]{nguyenvanpeski}. 

Let $n\in \Z_{\geq 1}$ and let $\Mat_n(\Zp)$ denote the set of $n\times n$ matrices over $\Zp$. For $M_1,\dots,M_k\in \Mat_n(\Zp)$, consider the flag
\begin{equation}
    \cok(M_1\cdots M_k) \onto \dots \onto \cok(M_1),
\end{equation}
where the surjection $\cok(M_1\dots M_{i+1})=\Zp^n/\im(M_1\dots M_{i+1})\onto \cok(M_1\dots M_i)=\Zp^n/\im(M_1\dots M_i)$ is induced by the inclusion $\im(M_1\dots M_{i+1}) \subeq \im(M_1\dots M_i)$. We denote this flag by $\bfcok(M_1,\dots,M_k)$. 

Since $\Mat_n(\Zp)\simeq \Zp^{n^2}$ is a compact topological group, there is a unique probability Haar measure on $\Mat_n(\Zp)$. We refer to a random element of $\Mat_n(\Zp)$ following the Haar measure as a \textbf{Haar-random matrix} in $\Mat_n(\Zp)$.

\subsection{Background}
When $M_1,\dots,M_k$ are random matrices in $\Mat_n(\Zp)$, we get a probability measure on $\Fl_k$ modeled by $\bfcok(M_1,\dots,M_k)$. In \cite{nguyenvanpeski}, Nguyen and Van Peski initiated the investigation of $\bfcok(M_1,\dots, M_k)$ by studying the joint distribution of $\cok(M_1), \dots, \allowbreak \cok(M_1\cdots M_k)$ as $\Zp$-modules. They proved a universality result in the sense of Wood \cite{wood2019random}, namely, if the $kn^2$ entries from $M_1,\dots,M_k$ are independent and each is not too concentrated mod $p$, then as $n\to \infty$, the limiting joint distribution of $\cok(M_1), \dots, \cok(M_1\cdots M_k)$ is insensitive to the exact distributions of these $kn^2$ entries. The limiting joint distribution is also explicitly determined.

Since the datum $(\cok(M_1), \dots, \cok(M_1\cdots M_k))$ is just $\bfcok(M_1,\dots,M_k)$ forgetting the surjections in between, it is natural to expect that the above results are explained by a universal distribution on the level of $\bfcok(M_1,\dots,M_k)\in \Fl_k$. Nguyen and Van Peski \cite[\S 10]{nguyenvanpeski} asked if the above results lift to $\bfcok(M_1,\dots,M_k)$; they defined the conjectured limiting distribution in \cite[Thm.~1.3]{nguyenvanpeski}. 

\subsection{Our result}
We answer their question affirmatively when $M_1,\dots,M_k$ distribute independently and follow the Haar measure. In this case, we also obtain the exact distribution of $\bfcok(M_1,\dots,M_k)$ for each fixed $n$. For $G\in \Mod(\Zp)$, let $r(G):=\dim_{\Fp} G/pG$ be the \textbf{rank} of $G$.

\begin{theorem}\label{thm:main}
    Fix $k\in \Z_{\geq 1}$. Let $M_1,\dots,M_k\in \Mat_n(\Zp)$ be independent and Haar-random, and fix $\mathbf{G}=(G_k\onto \dots\onto G_1)\in \Fl_k$ such that $\abs{G_k}<\infty$. Then for $n\geq r(G_k)$,
    \begin{equation}\label{eq:prob-haar}
        \Prob_{M_1,\dots,M_k\in \Mat_n(\Zp)}(\bfcok(M_1,\dots,M_k)\simeq \mathbf{G}) = \frac{1}{\abs{\Aut(\mathbf{G})}} \parens*{\prod_{i=n-r(G_k)+1}^n \!\! (1-p^{-i})} \parens*{\prod_{i=1}^n (1-p^{-i})}^k.
    \end{equation}
    In particular, when $n\to\infty$,
    \begin{equation}
        \lim_{n\to \infty}\Prob_{M_1,\dots,M_k\in \Mat_n(\Zp)}(\bfcok(M_1,\dots,M_k)\simeq \mathbf{G}) = \frac{1}{\abs{\Aut(\mathbf{G})}} \parens*{\prod_{i=1}^\infty (1-p^{-i})}^k.
    \end{equation}
\end{theorem}

\begin{remark*}
    It is clear that if $\abs{G_k}=\infty$ or $n<r(G_k)$, then the probability in \eqref{eq:prob-haar} is zero: $G_k=\cok(M_1\cdots M_k)$ must have rank at most $n$, and $\abs{G_k}=\infty$ only happens when at least one of $\det(M_1),\dots, \allowbreak\det(M_k)$ is zero, which happens with probability zero.
\end{remark*}

The flag $\bfcok(M_1,\dots,M_k)$ is probably the finest datum one could get from a chain of matrices $(M_1,\dots,M_k)$.\footnote{One could think of $(M_1,\dots,M_k)$ as a chain of linear maps $\Zp^n \overset{M_k}{\to} \dots \overset{M_1}{\to} \Zp^n$.} For example, $\cok(M_i M_{i+1}\dots M_j)$ for $1\leq i\leq j\leq k$ is isomorphic to the kernel of $G_j\onto G_{i-1}$, a concatenation of several surjections from $\bfcok(M_1,\dots,M_k)$. The distribution of $\bfcok(M_1,\dots,M_k)$ thus encodes the joint distribution of all $\cok(M_i M_{i+1}\dots M_j)$. We note that the joint distribution of certain subsets of these cokernels have natural connections to Hall algebras (see \cite[p.~46]{nguyenvanpeski}).

An important purpose of the paper is to demonstrate that some classical ideas to study the Cohen--Lenstra heuristic \cite{cohenlenstra1984heuristics, friedmanwashington1989} adapt to the (apparently highly refined) flag setting nicely. This point will be evident once we set up the language in \S \ref{subsec:dict}. The ease to work with flags will be showcased in our proof of Theorem \ref{thm:main}; for example, no knowledge about $\abs{\Aut(\mathbf{G})}$ is required.\footnote{Indeed when $k\geq 2$, we do not even have a combinatorial classification of $\mathbf{G}$ \cite[p.~44, Rmk.~7]{nguyenvanpeski}, let alone a general formula for $\abs{\Aut(\mathbf{G})}$. However, it is known that certain summations involving $1/\abs{\Aut(\mathbf{G})}$ are expressed in Hall--Littlewood polynomials \cite[\S 5--7,10]{nguyenvanpeski}.} In light of its simplicity, we conjecture that a suitable combination of our method and the general machinery of Sawin and Wood \cite{sawinwood} would yield a universality version of Theorem \ref{thm:main}. 

\subsection{Further applications}
A slight modification of the proof of Theorem \ref{thm:main} implies an analogue for non-square matrices. 

\begin{theorem}\label{thm:nonsquare}
    Fix $k\in \Z_{\geq 1}$, and $u_1,\dots,u_k\in \Z_{\geq 0}$. For $1\leq i\leq k$, let $M_i\in \Mat_{(n+u_{i-1})\times (n+u_i)}$, where $u_0:=0$, and assume $M_1,\dots,M_k$ are independent and Haar-random. Fix $\mathbf{G}=(G_k\onto \dots\onto G_0=0)\in \Fl_k$. Then
    \begin{equation}\label{eq:prob-nonsquare}
        \Prob_{M_1,\dots,M_k}(\bfcok(M_1,\dots,M_k)\simeq \mathbf{G}) = \frac{\prod_{j=1}^k \parens*{\frac{\abs{G_j}}{\abs{G_{j-1}}}}^{-u_j}}{\abs{\Aut(\mathbf{G})}} \parens*{\prod_{i=n-r(G_k)+1}^n \!\! (1-p^{-i})} \parens*{\prod_{j=1}^k \prod_{i=1}^{n} (1-p^{-i-u_j})}
    \end{equation}
    if $\abs{G_k}<\infty$ and $r(G_k)\leq n$, and zero otherwise. In particular, when $n\to \infty$, 
    \begin{equation}\label{eq:prob-nonsquare-limit}
        \lim_{n\to \infty} \Prob_{M_1,\dots,M_k}(\bfcok(M_1,\dots,M_k)\simeq \mathbf{G}) = \frac{\prod_{j=1}^k \parens*{\frac{\abs{G_j}}{\abs{G_{j-1}}}}^{-u_j}}{\abs{\Aut(\mathbf{G})}} \prod_{j=1}^k \prod_{i=1}^{\infty} (1-p^{-i-u_j})
    \end{equation}
    if $\abs{G_k}<\infty$, and zero otherwise.
\end{theorem}

\begin{remark*}
    It is less obvious \emph{a priori} why $\abs{\cok(M_1\cdots M_k)}<\infty$ with probability one, since we cannot use determinants anymore. This will be a consequence of our proof.
\end{remark*}

The non-flag case (i.e., $k=1$) was proved in \cite{wood2019random}, together with the universality result. We conjecture that \eqref{eq:prob-nonsquare-limit} holds if $M_1,\dots,M_k$ are independent and $\varepsilon$-balanced in the definition of \cite{wood2019random}.

As an application of the non-square analogue, we construct a $k$-parameter family of deformations of the ``Cohen--Lenstra probability measure'' in \eqref{eq:prob-haar}. For any flag $\mathbf{G}=(G_k\onto \dots\onto G_0=0)\in \Fl_k$ such that $\abs{G_k}<\infty$, we let $n_1(\mathbf{G}),\dots, n_k(\mathbf{G})\in \Z_{\geq 0}$ be defined by
\begin{equation}
    p^{n_i(\mathbf{G})}:=\abs{G_i}/\abs{G_{i-1}}.
\end{equation}

\begin{definition}\label{def:measure}
    For $k,n\in \Z_{\geq 1}, t_1,\dots,t_k\in [0,p)$, we define a measure $P_{n,(t_1,\dots,t_k)}$ on $\Fl_k$ by
    \begin{equation}\label{eq:measure-n}
        P_{n,(t_1,\dots,t_k)}(\mathbf{G}):=\frac{\prod_{j=1}^k t_j^{n_j(\mathbf{G})}}{\abs{\Aut(\mathbf{G})}}  \parens*{\prod_{i=n-r(G_k)+1}^n \!\! (1-p^{-i})} \parens*{\prod_{j=1}^k \prod_{i=1}^{n} (1-p^{-i}t_j)}
    \end{equation}
    if $\abs{G_k}<\infty$ and $r(G_k)\leq n$, and zero otherwise. Similarly, define
    \begin{equation}
        P_{\infty,(t_1,\dots,t_k)}(\mathbf{G}):=\frac{\prod_{j=1}^k t_j^{n_j(\mathbf{G})}}{\abs{\Aut(\mathbf{G})}}  \parens*{\prod_{j=1}^k \prod_{i=1}^{\infty} (1-p^{-i}t_j)}
    \end{equation}
    if $\abs{G_k}<\infty$, and zero otherwise.
\end{definition}

When $t_i=1$, the measure \eqref{eq:measure-n} reduces to \eqref{eq:prob-haar}; when $t_i=p^{-u_i}, u_i\in \Z_{\geq 0}$, the measure \eqref{eq:measure-n} reduces to \eqref{eq:prob-nonsquare}. This means that in these cases, \eqref{eq:measure-n} comes from a random matrix model. A brief argument in \S \ref{sec:measure} implies: 

\begin{corollary}\label{cor:normalize}
    $P_{n,(t_1,\dots,t_k)}$ and $P_{\infty,(t_1,\dots,t_k)}$ are probability measures on $\Fl_k$.
\end{corollary}

When $k=1$, Definition~\ref{def:measure} is precisely the measure considered by Fulman and Kaplan \cite{fulmankaplan2019}. The joint distribution of $(G_1,\dots,G_k)$ for $\mathbf{G}$ distributed according to \eqref{eq:measure-n} is expressible in terms of Hall--Littlewood polynomials and can be extracted from \cite{vanpeski2021limits}, see Proposition \ref{prop:joint}.

Part of the content of Corollary \ref{cor:normalize} is that the normalizing constant in Definition~\ref{def:measure} is correct. For example, $P_{\infty,(t_1,\dots,t_k)}$ being a probability measure is equivalent to the formal identity
\begin{equation}\label{eq:flag-cl}
    \sum_{\substack{\mathbf{G}\\\abs{G_k}<\infty}} \frac{\prod_{j=1}^k t_j^{n_j(\mathbf{G})}}{\abs{\Aut(\mathbf{G})}} = \prod_{j=1}^k \prod_{i=1}^\infty \frac{1}{1-p^{-i} t_j} \in \C[[t_1,\dots,t_k]].
\end{equation}
Its function field analogue naturally connects to matrices in a parabolic subalgebra; see \S \ref{subsec:matrix}.

Deformations of the Cohen--Lenstra measure are already interesting without motivations from random matrix models; see \cite{delaunayjouhet2014,fulmankaplan2019} and \cite[\S 7]{nguyenvanpeski} for the intrinsic study of some probability measures of this type, and note the ubiquity of Hall--Littlewood polynomials. But furthermore, such deformations also arise as predicted distributions of arithmetic objects: for example \cite{cohenlenstra1984heuristics}, if $p$ is odd and we take $k=1, u=u_1, G=\mathbf{G}=G_1$ in \eqref{eq:prob-nonsquare-limit}, then the resulting probability mass function, proportional to $1/(\abs{G}^u\abs{\Aut(G)})$, predicts the distribution of the $p$-part of the class group of a random quadratic extension of $\Q$ ($u=0$ for imaginary, $u=1$ for real). We propose the following question to conclude the introduction.
\begin{question}
    Are there arithmetic settings that naturally produce a flag of finite abelian groups? Is the distribution of its $p$-part predicted by one of the distributions above?
\end{question}

\subsection*{Acknowledgements} The author thanks Roger Van Peski for fruitful discussions and comments on earlier drafts.

\section{Preliminaries}
\subsection{Dictionary of flags}\label{subsec:dict}
Given $n,k\in \Z_{\geq 1}$ and a surjective flag $\mathbf{G}=(G_k\onto\dots\onto G_1)\in \Fl_k$, a \textbf{surjection} from $\Zp^n$ to $\mathbf{G}$ simply refers to a surjective $\Zp$-linear map $f_k:\Zp^n\onto G_k$. Any such surjection $f_k$ induces a chain of surjections
\begin{equation}
    \Zp^n \overset{f_k}{\onto} G_k \onto \dots \onto G_1 \onto G_0=0,
\end{equation}
which induces a surjection $f_i:\Zp^n\onto G_i$ for $0\leq i\leq k$ by composition. We can equivalently think of a surjection $f:\Zp^n\onto \mathbf{G}$ as a collection $(f_i:\Zp^n\onto G_i)_i$, but keeping in mind that $f_k$ determines the rest.

By a $k$-\textbf{injective flag} in $\Zp^n$, we mean a tower of $\Zp$-submodules of $\Zp^n$:
\begin{equation}
    F_k \subeq F_{k-1} \subeq \dots \subeq F_1 \subeq F_0:=\Zp^n.
\end{equation}
We denote by $\calFl_k(\Zp^n)$ the set of $k$-injective flags in $\Zp^n$. 

There is a one-to-one correspondence between $\calFl_k(\Zp^n)$ and the set of $k$-surjective flags $\mathbf{G}$ together with a surjection $\Zp^n\onto \mathbf{G}$. To $\mathcal{F}=(F_k\subeq \dots \subeq F_0=\Zp^n)$, we associate $\Zp^n\onto \mathbf{G}=(G_k\onto \dots \onto G_0=0)$, where $G_i=\Zp^n/F_i$ for $0\leq i\leq k$ and the maps are the natural quotient maps. Conversely, given $f: \Zp^n\onto \mathbf{G}$, we recover $\mathcal{F}$ by $F_i=\ker(f_i:\Zp^n\onto G_i)$ for $0\leq i\leq k$. We introduce the natural notation $\mathbf{G}=\Zp^n/\mathcal{F}$ and $\mathcal{F}=\itker(f:\Zp^n\onto \mathbf{G})$. 

Given matrices $M_1,\dots,M_k\in \Mat_n(\Zp)$, we define an injective flag by $F_i=\im(M_1\dots M_i)\subeq \Zp^n$ for $1\leq i\leq k$. We denote this flag by $\itim(M_1,\dots,M_k)$. We canonically have $\bfcok(M_1,\dots,M_k)\simeq \Zp^n/\itim(M_1,\dots,M_k)$. 

\section{Proof of Theorem \ref{thm:main}}
Fix $\mathbf{G}=(G_k\onto\dots\onto G_1)\in \Fl_k$ with $\abs{G_k}<\infty$. Let $r:=r(G_k)$, and fix $n\geq r$. We separate the major steps of proving Theorem \ref{thm:main} into the following lemmas.

\begin{lemma}\label{lem:quot-count}
    The number of $k$-injective flags $\mathcal{F}$ in $\Zp^n$ such that $\Zp^n/\mathcal{F}\simeq \mathbf{G}$ is
    \begin{equation}
        \frac{\abs{G_k}^n}{\abs{\Aut(\mathbf{G})}} \prod_{i=n-r(G_k)+1}^n (1-p^{-i}).
    \end{equation}
\end{lemma}

\begin{proof}
    Fix a copy of $\mathbf{G}$, and let $\Surj(\Zp^n,\mathbf{G})$ denote the set of surjections from $\Zp^n$ to $\mathbf{G}$, which is nothing but the set of surjections $\Zp^n\onto G_k$. By Nakayama's lemma,
    \begin{equation}
        \abs{\Surj(\Zp^n,\mathbf{G})} = \abs{G_k}^n \prod_{i=n-r+1}^n (1-p^{-i}).
    \end{equation}

    Let $\Aut(\mathbf{G})$ act on $\Surj(\Zp^n,\mathbf{G})$ by composition. As usual, the action is free: if $\sigma=(\sigma_i)\in \Aut(\mathbf{G})$ with $\sigma_i\in \Aut(G_i)$ is such that $\sigma_i \circ f_i=f_i$ for all $i$, then since $f_i$ is surjective, we must have $\sigma_i=\mathrm{id}$. As a consequence, the orbit space has cardinality given by
    \begin{equation}
        \abs*{\frac{\Surj(\Zp^n,\mathbf{G})}{\Aut(\mathbf{G})}} = \frac{\abs*{\Surj(\Zp^n,\mathbf{G})}}{\abs{\Aut(\mathbf{G})}} = \frac{\abs{G_k}^n}{\abs{\Aut(\mathbf{G})}} \prod_{i=n-r+1}^n (1-p^{-i}).
    \end{equation}

    Finally, it is easy to verify that the orbit space above is in a canonical bijection with the set of $k$-injective flags $\mathcal{F}$ with $\Zp^n/\mathcal{F}\simeq \mathbf{G}$. The conclusion then follows.
\end{proof}

\begin{lemma}\label{lem:prob-image}
    Fix a $k$-injective flag $\mathcal{F}$ such that $\Zp^n/\mathcal{F}\simeq \mathbf{G}$. Then if $M_1,\dots,M_k\in \Mat_n(\Zp)$ are independent and Haar-random, then
    \begin{equation}
        \Prob_{M_1,\dots,M_k\in\Mat_n(\Zp)}(\itim(M_1,\dots,M_k)=\mathcal{F}) = \abs{G_k}^{-n} \parens*{\prod_{i=1}^n (1-p^{-i})}^k.
    \end{equation}
\end{lemma}
\begin{proof}
    Let $\mathcal{F}=(F_k\subeq \dots \subeq F_1\subeq \Zp^n)$. Since $\abs{G_k}=\abs{\Zp^n/F_k}<\infty$, every $F_i$ is a free module of rank $n$. We first pick $M_1$ with the condition that $\im(M_1)=F_1$. This means two things: (1) $M_1: \Zp^n\to \Zp^n$ has image in $F_1$. (2) The induced $\Zp$-linear map $M_1: \Zp^n \to F_1$ is surjective. 
    
    The probability that $\im(M_1)\subeq F_1$ is the probability that every column of $M_1$ lies in $F_1$. Hence $\Prob_{M_1\in \Mat_n(\Zp)}(\im(M_1)\subeq F_1)$ is $\abs{\Zp^n/F_1}^{-n} = \abs{G_1}^{-n}$. By Nakayama's lemma, since $F_1$ is of rank $n$, the probability that a Haar-random linear map $\Zp^n \to F_1$ be surjective is $\prod_{i=1}^n (1-p^{-i})$. As a result,
    \begin{equation}
        \Prob_{M_1\in \Mat_n(\Zp)}(\im(M_1)=F_1)=\abs{G_1}^{-n} \prod_{i=1}^n (1-p^{-i}).
    \end{equation}

    Now we fix $M_1$ and pick $M_2$ with the condition that $\im(M_1 M_2)=F_2$. We note that $M_1:\Zp^n\to F_1$ is an isomorphism because it is a surjective map between rank $n$ free modules over $\Zp$. Therefore, $M_1$ induces an isomorphism of flags from $M_1^{-1}(F_2)\subeq \Zp^n$ to $F_2\subeq F_1$. Thus, $\im(M_1 M_2)=F_2$ if and only if $\im(M_2)=M_1^{-1}(F_2)\subeq \Zp^n$. By the same argument as above, and the fact that $\abs{\Zp^n/M_1^{-1}(F_2)}=\abs{F_1/F_2}=\abs{G_2}/\abs{G_1}$, we get
    \begin{equation}
        \Prob_{M_2\in \Mat_n(\Zp)}\left(\im(M_1 M_2)=F_2|\im(M_1)=F_1\right) = \parens*{\frac{\abs{G_2}}{\abs{G_1}}}^{-n} \prod_{i=1}^n (1-p^{-i}).
    \end{equation}

    Repeating the argument and multiplying all of the above probabilities together, we conclude that
    \begin{equation}
        \begin{aligned}
            &\Prob_{M_1,\dots,M_k\in\Mat_n(\Zp)}(\itim(M_1,\dots,M_k)=\mathcal{F}) \\
            &=  \abs{G_1}^{-n} \prod_{i=1}^n (1-p^{-i}) \cdot \parens*{\frac{\abs{G_2}}{\abs{G_1}}}^{-n} \prod_{i=1}^n (1-p^{-i}) \cdot \dots \cdot \parens*{\frac{\abs{G_k}}{\abs{G_{k-1}}}}^{-n} \prod_{i=1}^n (1-p^{-i}) \\
            &=\abs{G_k}^{-n} \parens*{\prod_{i=1}^n (1-p^{-i})}^k. 
        \end{aligned}\qedhere
    \end{equation}
\end{proof}

\begin{proof}
    [Proof of Theorem \ref{thm:main}]
    Since $\bfcok(M_1,\dots,M_k)\simeq \Zp^n/\itim(M_1,\dots,M_k)$, we have
    \begin{equation}
        \Prob_{M_1,\dots,M_k\in \Mat_n(\Zp)}(\bfcok(M_1,\dots,M_k)\simeq \mathbf{G}) = \sum_{\substack{\mathcal{F}\in \calFl_k(\Zp^n)\\ \Zp^n/\mathcal{F}\simeq \mathbf{G}}} \Prob_{M_1,\dots,M_k\in \Mat_n(\Zp)}(\itim(M_1,\dots,M_k)=\mathcal{F}).
    \end{equation}

    Since the probability in Lemma~\ref{lem:prob-image} depends only on $\mathbf{G}$ but not on $\mathcal{F}$, the proof is complete by multiplying the results of Lemmas~\ref{lem:quot-count} and \ref{lem:prob-image}. 
\end{proof}

\section{Proof of Theorem \ref{thm:nonsquare}}
We state and prove a convenient lemma first.

\begin{lemma}\label{lem:needed}
    Let $A,B,C$ be free modules over $\Zp$ of finite ranks $a,b,c$, and suppose $f:B\onto A$ is a surjective linear map. Then for a Haar-random linear map $g:C\to B$, we have
    \begin{equation}
        \Prob_{g\in \Hom(C,B)}(\im(fg)=A) = \prod_{i=c-a+1}^c (1-p^{-i}).
    \end{equation}
\end{lemma}
\begin{proof}
    Without loss of generality, we may assume $A=\Zp^a, B=\Zp^b, C=\Zp^c$ and $f:\Zp^b\to \Zp^a$ is the projection to the first $a$ coordinates. Write $g\in \Mat_{b\times c}(\Zp)$ as $g=\begin{bmatrix}
        g_1 \\
        g_2
    \end{bmatrix}$, where $g_1\in \Mat_{a\times c}(\Zp)$ and $g_2\in \Mat_{(b-a)\times c}(\Zp)$. Then $fg=g_1$. Thus, the probability that $fg$ be surjective is the probability that $g_1$ be surjective, which is $\prod_{i=c-a+1}^c (1-p^{-i})$ by Nakayama's lemma.
\end{proof}

To prove Theorem \ref{thm:nonsquare}, we follow the same argument as Theorem \ref{thm:main}, except that we need to prove a more general version of Lemma \ref{lem:prob-image}. Again, we fix $\mathbf{G}=(G_k\onto\dots\onto G_1)\in \Fl_k$, but we do not assume $\abs{G_k}<\infty$. Fix $\mathcal{F}\in\calFl_k(\Zp^n)$ such that $\Zp^n/\mathcal{F}\simeq \mathbf{G}$. For $1\leq i\leq k$, let $M_i\in \Mat_{(n+u_{i-1})\times (n+u_i)}$, where $u_0:=0$, and assume $M_1,\dots,M_k$ are independent and Haar-random.

\begin{lemma}\label{lem:prob-image-nonsquare}
    In the setting above, we have
    \begin{equation}\label{eq:prob-image-nonsquare}
        \Prob_{M_1,\dots,M_k\in\Mat_n(\Zp)}(\itim(M_1,\dots,M_k)=\mathcal{F}) = \abs{G_k}^{-n} \prod_{j=1}^k \parens*{\frac{\abs{G_j}}{\abs{G_{j-1}}}}^{-u_j}\cdot \prod_{j=1}^k \prod_{i=1}^{n} (1-p^{-i-u_j})
    \end{equation}
    if $\abs{G_k}<\infty$, and zero otherwise.
\end{lemma}
\begin{proof}
    Let $\mathcal{F}=(F_k\subeq \dots \subeq F_1\subeq \Zp^n)$. We first pick $M_1\in \Mat_{n\times (n+u_1)}$ with $\im(M_1)=F_1$. If $\abs{G_1}=\infty$, then $\abs{\Zp^n/F_1}=\infty$, so the probability that each column of $M_1$ be in $F_1$ is zero. Therefore, we may assume $\abs{G_1}<\infty$ from now on. As a result, $F_1$ is free of rank $n$. By the similar argument in the proof of Lemma \ref{lem:prob-image}, we get
    \begin{equation}\label{eq:step1}
        \Prob_{M_1\in \Mat_{n\times (n+u_1)}(\Zp)}(\im(M_1)=F_1)=\abs{G_1}^{-(n+u_1)} \prod_{i=u_1+1}^{n+u_1} (1-p^{-i})=\abs{G_1}^{-(n+u_1)} \prod_{i=1}^{n} (1-p^{-i-u_1}).
    \end{equation}

    Now we fix $M_1$ and pick $M_2\in \Mat_{(n+u_1)\times (n+u_2)}$ with the condition that $\im(M_1 M_2)=F_2$, which is equivalent to $\im(M_2)\subeq M_1^{-1}(F_2)$. The third isomorphism theorem applied to the surjection $M_1:\Zp^{n+u_1}\to F_1$ gives an isomorphism $\Zp^{n+u_1}/M_1^{-1}(F_2) \simeq F_1/F_2$, so $\abs{\Zp^{n+u_1}/M_1^{-1}(F_2)} =\abs{F_1/F_2}=\abs{G_2}/\abs{G_1}$. Thus, 
    \begin{equation}\label{eq:needed}
        \Prob_{M_2\in \Mat_{(n+u_1)\times (n+u_2)}}(\im(M_1M_2)\subeq F_2 | \im(M_1)=F_1) = \parens*{\frac{\abs{G_2}}{\abs{G_1}}}^{-(n+u_2)}
    \end{equation}
    if $\abs{G_2}<\infty$, and zero otherwise. So again, we may assume $\abs{G_2}<\infty$ from now on.
    
    We now find the probability that $\im(M_1 M_2)=F_2$ conditioned on $\im(M_1 M_2)\subeq F_2$, so $M_2$ is a Haar-random linear map in $\Hom(\Zp^{n+u_2},M_1^{-1}(F_2))$. Note that we are in the setting of Lemma \ref{lem:needed} with $A=F_2$, $B=M_1^{-1}(F_2)$, $C=\Zp^{n+u_2}$, $f=M_1:M_1^{-1}(F_2)\onto F_2$, and $g=M_2: \Zp^{n+u_2}\to M_1^{-1}(F_2)$. The condition that $\im(M_1 M_2)=F_2$ is equivalent to $\im(fg)=A$. Since $\abs{\Zp^n/F_2}=\abs{G_2}<\infty$, $A$ is free of rank $n$. Since $\abs{\Zp^{n+u_1}/B}=\abs{G_2}/\abs{G_1}<\infty$, $B$ is free of rank $n+u_1$. By Lemma \ref{lem:needed} with $a=n$ and $c=n+u_2$, we get
    \begin{equation}
        \Prob_{M_2\in \Mat_{(n+u_1)\times (n+u_2)}}(\im(M_1M_2)=F_2 | \im(M_1)=F_1, \im(M_1M_2)\subeq F_2) = \prod_{i=u_2+1}^{n+u_2} (1-p^{-i}).
    \end{equation}
    Combined with \eqref{eq:needed}, we get
    \begin{equation}\label{eq:step2}
        \Prob_{M_2\in \Mat_{(n+u_1)\times (n+u_2)}}(\im(M_1M_2)=F_2 | \im(M_1)=F_1) = \parens*{\frac{\abs{G_2}}{\abs{G_1}}}^{-(n+u_2)} \prod_{i=1}^n (1-p^{-i-u_2}).
    \end{equation}

    Repeating the argument inductively and multiplying the probabilities in \eqref{eq:step1}, \eqref{eq:step2}, and so on, the desired formula \eqref{eq:prob-image-nonsquare} follows, along with the fact that each $G_i$ must be finite in order for the probability to be nonzero. 
\end{proof}

\begin{proof}
    [Proof of Theorem \ref{thm:nonsquare}]
    Multiply the results of Lemma \ref{lem:quot-count} and Lemma \ref{lem:prob-image-nonsquare}.
\end{proof}

\section{Probability measures on flags}\label{sec:measure}
Here we prove Corollary \ref{cor:normalize} as an immediate consequence of Theorem \ref{thm:nonsquare}.

\begin{proof}
    [Proof of Corollary \ref{cor:normalize}]
    It is clear that $P_{n,(t_1,\dots,t_k)}$ is a nonnegative measure, so it suffices to show that $\sum_{\mathbf{G}\in \Fl_k} P_{n,(t_1,\dots,t_k)}(\mathbf{G})=1$. We note that the measure in \eqref{eq:prob-nonsquare} is precisely $P_{n,(p^{-u_1},\dots,p^{-u_k})}$. In particular, Theorem \ref{thm:nonsquare} implies that $P_{n,(p^{-u_1},\dots,p^{-u_k})}$ is a probability measure for every $u_1,\dots,u_k\in \Z_{\geq 0}$, so the formal power series $\sum_{\mathbf{G}\in \Fl_k} P_{n,(t_1,\dots,t_k)}(\mathbf{G})$ in $t_1,\dots,t_k$ must be 1, completing the proof. 
    
    The case of $P_{\infty,(t_1,\dots,t_k)}$ is similar.
\end{proof}

Write $\mathbf{G}=(G_k\onto\dots \onto G_1)$. We compute the joint distribution of $(G_1,\dots,G_k)$ if $\mathbf{G}$ is distributed according to \eqref{eq:measure-n}. For partitions $\lambda^{(1)},\dots,\lambda^{(k)}$, we write $\mathbf{G}\sim (\lambda^{(1)},\dots,\lambda^{(k)})$ if the type of $G_i$ is $\lambda^{(i)}$. 

\begin{proposition}\label{prop:joint}
    Fixing partitions $\lambda^{(1)},\dots,\lambda^{(k)}$. For $P_{n,(t_1,\dots,t_k)}(\mathbf{G})$ as in \eqref{eq:measure-n}, the quantity
    \begin{equation}
        \sum_{\mathbf{G}\sim (\lambda^{(1)},\dots,\lambda^{(k)})} P_{n,(t_1,\dots,t_k)}(\mathbf{G})
    \end{equation}
    is given by \cite[p.~19, Prop.~2.6]{vanpeski2021limits} with $q=0$, $a^{(i)}_j=t_i p^{-j}$, $\mathbf{b}=(1,p^{-1},\dots,p^{-(n-1)})$, and $t=1/p$ in their notation.
\end{proposition}
\begin{proof}
    If $t_i=1$, by Theorem \ref{thm:main}, the quantity in question is the probability that $\cok(M_1\dots M_i)$ is of type $\lambda^{(i)}$ for all $i$, where $M_1,\dots,M_k \in \Mat_n(\Zp)$ are independent and Haar-random. This is given by \cite[p.~27, Cor.~3.4]{vanpeski2021limits} with $N_i=\infty$ in their notation. 

    To go from the $t_i=1$ case to the general case, we notice that if we fix $\varnothing=\lambda^{(0)},\lambda^{(1)},\dots,\lambda^{(k)}$ and letting $n_i=\abs{\lambda^{(i)}}-\abs{\lambda^{(i-1)}}$ for $1\leq i\leq k$, then from \eqref{eq:measure-n}, we have
    \begin{equation}
        \frac{\sum_{\mathbf{G}\sim (\lambda^{(1)},\dots,\lambda^{(k)})} P_{n,(t_1,\dots,t_k)}(\mathbf{G})}{\sum_{\mathbf{G}\sim (\lambda^{(1)},\dots,\lambda^{(k)})} P_{n,(1,\dots,1)}(\mathbf{G})} = \prod_{j=1}^k \parens*{t_j^{n_j}\prod_{i=1}^n \frac{1-p^{-i}t_j}{1-p^{-i}}}.
    \end{equation}
    Combining this with the $t_i=1$ case above finally gives the desired formula. 
\end{proof}

\subsection{Relation to matrices over finite fields}\label{subsec:matrix}
We give a direct proof of Proposition \ref{prop}, a function field analogue of \eqref{eq:flag-cl}, by establishing \eqref{eq:flag-comm} that connects it to counting matrices over finite fields. Let $\Fq$ be the finite field with $q$ elements, and let $R$ denote the power series ring $\Fq[[T]]$. We can similarly define $\Fl_k(R)$ to be the set of $k$-surjective flags of $R$-modules. Given $1\leq i\leq k$ and any flag $\mathbf{G}=(G_k\onto \dots \onto G_0=0)\in \Fl_k(R)$ such that $\dim_{\Fq} G_k<\infty$, let $n_i(\mathbf{G}):=\dim_{\Fq} G_i -\dim_{\Fq} G_{i-1}$. Define the \textbf{flag Cohen--Lenstra series} of $R$ as
\begin{equation}\label{eq:flag-cl-def}
    \Zhat_{R}(t_1,\dots,t_k):=\sum_{\substack{\mathbf{G}\in \Fl_k(R)\\ \dim G_k<\infty}} \frac{\prod_{j=1}^k t_j^{n_j(\mathbf{G})}}{\abs{\Aut(\mathbf{G})}}\in \Q[[t_1,\dots,t_k]].
\end{equation}
When $k=1$, this construction is precisely the Cohen--Lenstra series defined in \cite{huang2023mutually}.

\begin{proposition}\label{prop}
    We have 
    \begin{equation}
        \Zhat_{\Fq[[T]]}(t_1,\dots,t_k) = \prod_{j=1}^k \prod_{i=1}^\infty \frac{1}{1-q^{-i} t_j}.
    \end{equation}
\end{proposition}

\begin{proof}
    To specify a $k$-surjective flag $\mathbf{G}$ of $\Fq[[T]]$-modules with given $n_i(\mathbf{G})=n_i$, it suffices to specify the $k$-surjective flag of underlying $\Fq$-vector spaces
    \begin{equation}
        V_k \overset{\phi_{k-1}}{\onto} \dots \overset{\phi_1}{\onto} V_1\onto V_0=0,
    \end{equation}
    together with $\Fq[[T]]$-module structures on $V_i$ that are compatible with the flag. Up to isomorphism of flags of $\Fq$-vector spaces, we may assume $V_i=\Fq^{n_1+\dots+n_i}$, and the surjection $\phi_{i}$ is the projection to the first $n_1+\dots+n_i$ coordinates. Then the compatible $\Fq[[T]]$-module structures on $V_i$ are determined by a nilpotent endomorphism $M$ on $V_k$ (as the multiplication by $T$ map), such that $M$ factors through an endomorphism on $V_i$ for each $i$. If $W_i:=\ker(V_k\onto V_i)$ for $0\leq i\leq k$, then this simply means $MW_i\subeq W_i$ for $0\leq i\leq k$. Since $W_i$ is the span of the last $n_{i+1}+\dots+n_k$ basis vectors in $V_k=\Fq^{n_1+\dots+n_k}$, this happens if and only if $M$ is \textbf{block-lower-triangular} with respect to the block sizes $n_1,\dots,n_k$ in both rows and columns. We denote the set of such block-lower-triangular matrices by $\Mat_{n_1,\dots,n_k}(\Fq)$.

    Let $\GL_{n_1,\dots,n_k}(\Fq)$, $\Nilp_{n_1,\dots,n_k}(\Fq)$ be the set of invertible \textit{resp.} nilpotent matrices in $\Mat_{n_1,\dots,n_k}(\Fq)$. By the discussion above, $k$-surjective flags of $\Fq[[T]]$-modules with a given dimension vector $(n_i)$ are parametrized by $\Nilp_{n_1,\dots,n_k}(\Fq)$. The group $\GL_{n_1,\dots,n_k}(\Fq)$ acts on $\Nilp_{n_1,\dots,n_k}(\Fq)$ by conjugation. The orbits correspond to isomorphism classes of flags of $\Fq[[T]]$-modules, and the stabilizers correspond to automorphisms of flags of $\Fq[[T]]$-modules. By a standard argument involving the orbit-stabilizer theorem, we get
    \begin{equation}\label{eq:flag-comm}
        \Zhat_{\Fq[[T]]}(t_1,\dots,t_k) = \sum_{n_1,\dots,n_k\geq 0} \frac{\abs{\Nilp_{n_1,\dots,n_k}(\Fq)}}{\abs{\GL_{n_1,\dots,n_k}(\Fq)}} t_1^{n_1}\dots t_k^{n_k}.
    \end{equation}

    Since a matrix in $\Mat_{n_1,\dots,n_k}(\Fq)$ is invertible \textit{resp.} nilpotent if and only if every diagonal block is invertible \textit{resp.} nilpotent, the generating series simplifies to
    \begin{equation}
        \Zhat_{\Fq[[T]]}(t_1,\dots,t_k) = \prod_{j=1}^k \sum_{n_j\geq 0} \frac{\abs{\Nilp_{n_j}(\Fq)}}{\abs{\GL_{n_j}(\Fq)}} t_j^{n_j}.
    \end{equation}

    But it is well-known that the sum equals to $\prod_{i=1}^\infty \frac{1}{1-q^{-i} t_j}$; for instance, we may use the theorem $\abs{\Nilp_n(\Fq)}=q^{n^2-n}$ of Fine and Herstein \cite{fineherstein1958}, and then apply an identity of Euler \cite[Eq.~(2.2.5)]{andrewspartitions}. This completes the proof.
\end{proof}

\begin{remark*}
    Of course, $\Mat_{n_1,\dots,n_k}(\Fq)$ and $\GL_{n_1,\dots,n_k}(\Fq)$ are nothing but the parabolic subalgebra and the parabolic subgroup of a suitable partial flag variety of type A. As is made apparent by Equation \eqref{eq:flag-comm}, the construction \eqref{eq:flag-cl-def} fits naturally into the study of commuting varieties of parabolic subalgebras; we refer the readers to \cite{buloisevain2016} for a geometric aspect of such research. 
\end{remark*}

\bibliography{flag-cohen-lenstra-arxiv.bbl}
   
\end{document}